\documentclass{amsart}
\usepackage{amsmath, amssymb, amsthm}
\usepackage{amscd}
\usepackage{graphicx} 
\usepackage{enumerate}

\title{On totally real numbers and equidistribution}
  \author[Fili]{Paul Fili}
 \address{Department of Mathematics\\ University of Rochester, Rochester, NY 14627}
 \email{fili@math.rochester.edu}
 \author[Miner]{Zachary Miner}
\address{Department of Mathematics\\ University of Texas at Austin, TX 78712}
\email{zminer@math.utexas.edu}
\date{\today}

\usepackage[pdftitle={On totally real numbers and equidistribution},pdfauthor={Fili and Miner},pdfstartview={}]{hyperref}

\usepackage[all]{xy}

\newtheorem{thm}{Theorem}

\newtheorem*{thm*}{Theorem}
\newtheorem*{alg*}{Algorithm}
\newtheorem*{lemma*}{Lemma}

\theoremstyle{remark}
\newtheorem{rmk}[thm]{Remark}
\newtheorem*{rmk*}{Remark}

\newtheorem*{notation*}{Notation}

\theoremstyle{definition}

\newtheorem*{defn*}{Definition}

\newcommand{\mybf}{\mathbb}

\newcommand{\bP}{\mybf{P}}
\newcommand{\bR}{\mybf{R}}

\newcommand{\bC}{\mybf{C}}

\newcommand{\bZ}{\mybf{Z}}
\newcommand{\bQ}{\mybf{Q}}

\newcommand{\al}{\alpha}

\providecommand{\abs}[1]{\lvert#1\rvert}

\newcommand{\p}{\partial}

\newcommand{\Qbar}{\overline{\mybf{Q}}}

\def\talltareesidedbox#1{\setbox0=\hbox{$#1$}\dimen0=\wd0 \advance\dimen0 by3pt\rlap{\hbox{\vrule height10pt width.4pt
 depth2pt \kern-.4pt\vrule height10.4pt width\dimen0 depth-10pt\kern-.4pt \vrule height10pt width.4pt depth2pt}}
 \relax \hbox to\dimen0{\hss$#1$\hss}}
\def\tareesidedbox#1{\setbox0=\hbox{$#1$}\dimen0=\wd0 \advance\dimen0 by3pt\rlap{\hbox{\vrule height8pt width.4pt
 depth2pt \kern-.4pt\vrule height8.4pt width\dimen0 depth-8pt\kern-.4pt \vrule height8pt width.4pt depth2pt}}
\relax \hbox to\dimen0{\hss$#1$\hss}}
\def\shorttareesidedbox#1{\setbox0=\hbox{$#1$}\dimen0=\wd0 \advance\dimen0 by3pt\rlap{\hbox{\vrule height7pt width.4pt
 depth2pt \kern-.4pt\vrule height7.4pt width\dimen0 depth-7pt\kern-.4pt \vrule height7pt width.4pt depth2pt}}
 \relax \hbox to\dimen0{\hss$#1$\hss}}

\newcommand{\Lip}{\operatorname{Lip}}

\begin{document}

\begin{abstract}
 C.J. Smyth and later Flammang studied the spectrum of the Weil height in the field of all totally real numbers, establishing both lower and upper bounds for the limit infimum of the height of all totally real integers and determining isolated values of the height. We remove the hypothesis that we consider only integers and establish an lower bound on the limit infimum of the height for all totally real numbers. Our proof relies on a quantitative equidistribution theorem for numbers of small height.
\end{abstract}

\thanks{The authors would like to thank Robert Rumely and the University of Georgia for organizing the 2011 VIGRE SSP in arithmetic dynamics and for their hospitality and financial support, during which time much of this research was conducted.}

\maketitle

Recall that an algebraic number is said to be \emph{totally real} if all of its Galois conjugates lie in the field $\bR$ under any choice of embedding $\Qbar\hookrightarrow\bC$. The totally real numbers form a field which we denote $\bQ^{\mathrm{tr}}$. 
Schinzel \cite{SchinzelTotReal} established a lower bound on the infimum of the nonzero values of the absolute logarithmic Weil height $h$ for all totally real numbers:\footnote{In fact Schinzel's result  originally implied this result for integers, however, it is easy to see that it generalizes to totally real nonintegers as well; cf. e.g., \cite{GarzaRealConjugates, HohnOnGarza}.}
\begin{thm*}[Schinzel 1973]
Let $\al\in\bQ^{\mathrm{tr}}$, $\al\neq 0,\pm 1$ be a totally real number. Then 
\[
 h(\al) \geq h\left(\frac{1+\sqrt{5}}{2}\right) = \frac{1}{2} \log \frac{1+\sqrt{5}}{2}=0.2406059\ldots
\]
where $h$ denotes the absolute logarithmic Weil height.
\end{thm*}
Work of Smyth \cite{SmythTotRealI,SmythTotRealII}, later improved by Flammang \cite{Flammang}, established bounds on the limit infimum of the height for the ring of totally real integers:
\begin{thm*}[Smyth 1980, Flammang 1996]
 Let $\bZ^{\mathrm{tr}}$ denote the ring of all totally real integers. Then
 \[
  \liminf_{\al\in \bZ^{\mathrm{tr}}} h(\al) \geq \frac{1}{2}\log 1.720566=0.271327\ldots
 \]
 and there are precisely six isolated values of the Weil height of a totally real integer in the interval $(0,0.271327\ldots)$. Further,
 \[
  \liminf_{\al\in \bZ^{\mathrm{tr}}} h(\al) \leq \log 1.31427\ldots = 0.27328\ldots
 \]
\end{thm*}
In this paper we establish the following bound for the limit infimum of the height which removes the hypothesis that we only consider integers and is instead a bound over all totally real numbers. Specifically, we prove the following:
\begin{thm}\label{thm:totreal}
 Let $\bQ^\mathrm{tr}$ denote the field of all totally real numbers. Then
 \[
  \liminf_{\al\in \bQ^\mathrm{tr}} h(\al) \geq \frac{140}{3}\cdot \left(\frac{1}{8}-\frac{1}{6\pi}\right)^2 = 0.241573\ldots 
 \]
\end{thm}

Previously it seems to have been unknown if in fact the limit infimum for the Weil height of all totally numbers in fact coincided with infimum of the nonzero values set by Schinzel. Our result may be thought of as similar in spirit to the result of Zagier on the spectrum of the height on the curve $y=x+1$ over $\Qbar$, that is, of $h(\al)+h(1-\al)$ (see \cite[\S 3.A]{Zagier}), and the bounds of Bombieri and Zannier for the limit infimum of the height of totally $p$-adic numbers \cite{BombieriZannierNote}.

Our proof relies on results about the equidistribution of points of small Weil height. Early work on the distribution of numbers by height included results like the uniform distribution of the Farey fractions, culminating in quantitative generalization of the distribution of rational points at all places by Choi \cite{ChoiThesis,ChoiDistribution}. Shortly after Choi's work, Bilu \cite{Bilu} formulated an equidistribution for algebraic points of small height at the archimedean place, proving that the Galois conjugates of a sequence of numbers with Weil height tending to zero must equidistribute along the unit circle in $\bC$ in the sense of weak convergence of measures. Later this theorem was made quantitative by Petsche \cite{PetscheQuantBilu}, and then vastly generalized by Favre and Rivera-Letelier \cite{FRL} (see also the corrigendum \cite{FRLcorrigendum}). We will use the result of Favre and Rivera-Letelier in our proof.

We note that of course the upper bound for the limit infimum set by Smyth remains valid for all totally real numbers. It is interesting to note that Smyth established this upper bound by constructing an explicit sequence of totally real algebraic integers with height limiting to the value $0.27328\ldots$ While Smyth's work precedes the development of arithmetic dynamics, his construction can naturally be thought of as dynamical in spirit, as the sequence he constructs is a certain sequence of preperiodic points for the rational map
\[
 H(x)=x-1/x
\]
which has a totally real archimedean Julia set. As these points have $H$-canonical height $0$, it follows by work of Petsche, Szpiro, and Tucker \cite[Theorem 1 and Prop. 16]{PST-pairing} that their standard Weil height must tend to a limit which is given by the the value of the Arakelov-Zhang pairing of $H$ with the usual squaring map:
\[
 \langle H(z),z^2\rangle_{\mathrm{AZ}} = \int_\bR \log^+\abs{x} \,d\mu_H(x) = 0.27328\ldots
\]
where $\mu_H$ is the canonical measure associated to iteration of $H(x)$. Smyth computed this value by determining a remarkable iterative formula for the distribution of $\mu_H$ (see \cite[Theorem 3]{SmythTotRealI}).

Before beginning our proof, we first recall some background in order to set up the quantitative equidistribution theorem. Suppose that $f\in C^1(\bP^1(\bC))$. We denote the Lipschitz constant of $f$ with respect to the chordal metric by
\[
 \Lip(f) = \sup_{x,y\in \bP^1(\bC)} \frac{\abs{f(x)-f(y)}}{ \mathsf{d}(x,y) },
\]
where
\[
 \mathsf{d}(x,y) = \frac{ \abs{x-y} }{\sqrt{1+\abs{x}^2}\cdot \sqrt{1+\abs{y}^2}}
\]
denotes the archimedean chordal metric on $\bP^1(\bC)$. Write $df = \p f + \bar\p f$ where $\p f$ is a form of type $(1,0)$ on $\bC=\bR\oplus i\bR$ and $\bar\p f$ is of type $(0,1)$, and define $d^c f = \frac{1}{2\pi i}(\p f - \bar\p f)$. (This normalization is to ensure that the Laplacian $\Delta g = dd^c g$ satisfies $\Delta g_\rho = \rho$ where $g_\rho(x) = \int_\bC \log\abs{x-y}\,d\rho(y)$, for suitable Borel probability measures $\rho$ on $\bC$.) For $f,g\in C^1(\bP^1(\bC))$ real-valued functions, we define the Dirichlet form to be
\[
 \langle f,g\rangle = \int_{\bC} df\wedge d^cg = \int_{\bC}\left( \frac{\p f}{\p x}\frac{\p g}{\p x} + \frac{\p f}{\p y}\frac{\p g}{\p y}\right)\frac{dx\wedge dy}{2\pi}.
\]
This defines a bilinear form which satisfies the Cauchy-Schwartz inequality $
 \langle f,g\rangle \leq \langle f,f\rangle^{1/2} \langle g,g\rangle^{1/2}.
$

Using the fact that our difference with the standard archimedean measure is 0, the H\"older constant $\kappa$ in \cite[Theorem 7]{FRL} can be taken to be $1$, and we arrive at the following statement:
\begin{thm*}[Quantitative equidistribution at $\bC$]
 Let $\lambda$ denote the normalized Haar measure of the unit circle in the complex plane and $G$ denote the absolute Galois group. There exists a constant $c>0$ such that for all $f\in C^1(\bP^1(\bC))$ and $\al\in\Qbar$,
\[
 \bigg| \frac{1}{d} \sum_{z\in G\al} f(z) - \int_{\bP^1(\bC)} f(z)\,d\lambda \bigg|
\leq \frac{\Lip(f)}{d} + \left( h(\al) + c \frac{\log d}{d} \right)^{1/2}\langle f,f\rangle^{1/2}
\]
where $d = \# G\al$ is the degree of $\al$ and $\Lip(f)$ denotes the Lipschitz constant of $f$ on $\bP^1(\bC)$ as defined above.
\end{thm*}

\begin{proof}[Proof of Theorem \ref{thm:totreal}]
We will obtain our lower bound in Theorem \ref{thm:totreal} by integrating a test function $f(z)$ and computing the exact values of the integral of $f$ against $\lambda$ and the average of $f$ over the Galois conjugates of any totally real number. Specifically, let
\[
 f(z) = \mathsf{d}(z,i)^3\cdot \mathsf{d}(z,-i)^3 = \left( \frac{\sqrt{x^2+(y-1)^2}\sqrt{x^2 + (y+1)^2}}{2(1+x^2+y^2)}\right)^3
\]
where $z=x+iy$ and as defined above $\mathsf{d}$ denotes the chordal metric on $\bP^1(\bC)$. Notice that $f$ is $C^1$ on $\bC=\bR\oplus i\bR\cong \bR^2$ since both partial derivatives are continuous. Further, for $x\in \bR$, we have $d(x,i)=d(x,-i)=1/\sqrt{2}$, and so $f(x) = 1/8$. Thus, for \emph{any} totally real $\al$, the average
\[
 \frac{1}{d} \sum_{z\in G\al} f(z) = \frac{1}{8}.
\]
Further, it is a simple matter to compute that
\[
 \int_{\bP^1(\bC)} f(z)\,d\lambda(z) = \frac{1}{6\pi}.
\]
Thus by quantitative equidistribution we have
\[
 \left|\frac{1}{8}-\frac{1}{6\pi}\right| \leq \frac{\Lip(f)}{d} + \left( h(\al) + c \frac{\log d}{d} \right)^{1/2}\langle f,f\rangle^{1/2}.
\]
Notice that as we take the limit infimum as $\al$ ranges over all elements of $\bQ^{\mathrm{tr}}$, we only care about numbers of a bounded height, so by Northcott's theorem the degree $d$ must tend to infinity. Thus when we take the limit infimum above, we obtain the lower bound:
\[
 \left|\frac{1}{8}-\frac{1}{6\pi}\right| \leq \liminf_{\al\in \bQ^{\mathrm{tr}}} h(\al)^{1/2} \cdot \langle f,f\rangle^{1/2}
\]
With the aid of a computer it is easy to calculate the exact value $\langle f,f\rangle = 3/140$, and thus squaring both sides and rearranging terms, we arrive at
\[
 \frac{140}{3}\cdot \left|\frac{1}{8}-\frac{1}{6\pi}\right|^2 = 0.241573\ldots \leq \liminf_{\al\in \bQ^{\mathrm{tr}}} h(\al)
\]
which concludes the proof of Theorem \ref{thm:totreal}.
\end{proof}

\begin{rmk}
 This result can be improved somewhat, for example, by optimizing the choice of exponent (for example, choosing $3.3$ instead of $3$ will result in a bound of $0.241713\ldots$, at the cost of complicating the computations), however, it seems unlikely that this can be much improved without a somewhat different approach that involves dropping the assumption that $f$ is constant on the real line. More specifically, the main term of the quantitative equidistribution theorem quoted above essentially arises from viewing the integral of $f$ against a smoothed approximation to the measure $\lambda - [\al]$, where $[\al]$ denotes the probability measure equally distributed among the Galois conjugates of $\al$, as a Dirichlet pairing between $f$ and the potential function $g$ of an approximation to $\lambda - [\al]$; consequently, we expect the inequality to be strict, and our lower bound not to be sharp, whenever there do not exist constants $c,d$ such that $f\neq cg+d$. Thus an `ideal' choice of $f$ might be such a potential function; however, such a function is almost certain not to be constant on $\bP^1(\bR)$, which was a key assumption in our approach.
\end{rmk}

\bibliographystyle{abbrv} 
\bibliography{bib}        

\begin{thebibliography}{10}

\bibitem{Bilu}
Y.~Bilu.
\newblock Limit distribution of small points on algebraic tori.
\newblock {\em Duke Math. J.}, 89(3):465--476, 1997.

\bibitem{BombieriZannierNote}
E.~Bombieri and U.~Zannier.
\newblock A note on heights in certain infinite extensions of {$\mathbb{Q}$}.
\newblock {\em Atti Accad. Naz. Lincei Cl. Sci. Fis. Mat. Natur. Rend. Lincei
  (9) Mat. Appl.}, 12:5--14, 2001.

\bibitem{ChoiThesis}
K.-K. Choi.
\newblock Diophantine approximation on projective spaces over number fields.
\newblock {\em Ph.D. dissertation, University of Texas at Austin}, 1996.

\bibitem{ChoiDistribution}
K.-K. Choi.
\newblock On the distribution of points in projective space of bounded height.
\newblock {\em Trans. Amer. Math. Soc.}, 352(3):1071--1111, 2000.

\bibitem{FRL}
C.~Favre and J.~Rivera-Letelier.
\newblock \'{E}quidistribution quantitative des points de petite hauteur sur la
  droite projective.
\newblock {\em Math. Ann.}, 335(2):311--361, 2006.

\bibitem{FRLcorrigendum}
C.~Favre and J.~Rivera-Letelier.
\newblock Corrigendum to: ``{Q}uantitative uniform distribution of points of
  small height on the projective line'' ({F}rench) [{M}ath. {A}nn. {\bf 335}
  (2006), no. 2, 311--361; mr2221116].
\newblock {\em Math. Ann.}, 339(4):799--801, 2007.

\bibitem{Flammang}
V.~Flammang.
\newblock Two new points in the spectrum of the absolute {M}ahler measure of
  totally positive algebraic integers.
\newblock {\em Math. Comp.}, 65(213):307--311, 1996.

\bibitem{GarzaRealConjugates}
J.~Garza.
\newblock On the height of algebraic numbers with real conjugates.
\newblock {\em Acta Arith.}, 128(4):385--389, 2007.

\bibitem{HohnOnGarza}
G.~H{\"o}hn.
\newblock On a theorem of {G}arza regarding algebraic numbers with real
  conjugates.
\newblock {\em Int. J. Number Theory}, 7(4):943--945, 2011.

\bibitem{PetscheQuantBilu}
C.~Petsche.
\newblock A quantitative version of {B}ilu's equidistribution theorem.
\newblock {\em Int. J. Number Theory}, 1(2):281--291, 2005.

\bibitem{PST-pairing}
C.~{Petsche}, L.~{Szpiro}, and T.~J. {Tucker}.
\newblock {A dynamical pairing between two rational maps}.
\newblock {\em Trans. Amer. Math. Soc.}, 364(4):1687--1710, 2012.

\bibitem{SchinzelTotReal}
A.~Schinzel.
\newblock On the product of the conjugates outside the unit circle of an
  algebraic number.
\newblock {\em Acta Arith.}, 24:385--399, 1973.
\newblock Collection of articles dedicated to Carl Ludwig Siegel on the
  occasion of his seventy-fifth birthday. IV.

\bibitem{SmythTotRealI}
C.~J. Smyth.
\newblock On the measure of totally real algebraic integers.
\newblock {\em J. Austral. Math. Soc. Ser. A}, 30(2):137--149, 1980/81.

\bibitem{SmythTotRealII}
C.~J. Smyth.
\newblock On the measure of totally real algebraic integers. {II}.
\newblock {\em Math. Comp.}, 37(155):205--208, 1981.

\bibitem{Zagier}
D.~Zagier.
\newblock Algebraic numbers close to both {$0$} and {$1$}.
\newblock {\em Math. Comp.}, 61(203):485--491, 1993.

\end{thebibliography}

\end{document}